\documentclass[11pt]{amsart}
\usepackage{amsfonts,amsmath,amssymb,amsthm}
\usepackage[latin1]{inputenc}
\theoremstyle{plain} 
\newtheorem{lemma}{Lemma}
\newtheorem{theorem}{Theorem}

\newtheorem{corollary}{Corollary}
\newcommand{\R}{\mathbb{R}}
\newcommand{\N}{\mathbb{N}}

\theoremstyle{definition}

\newtheorem{example}{Example}

\theoremstyle{remark}

\numberwithin{equation}{section}
\let\w=\omega
\let\e=\epsilon

\renewcommand{\ll}{\tilde{\lambda}}

\newcommand{\lx}[2]{\lambda(#2,\T^{#1})}
\newcommand{\hx}[2]{h(#2,\T^{#1})}

\newcommand{\hone}[1]{\hx{}{#1}}
\newcommand{\hn}[1]{\hx{n}{#1}}
\newcommand{\lone}[1]{\lx{}{#1}}
\renewcommand{\ln}[1]{\lx{n}{#1}}

\newcommand{\D}{D}
\newcommand{\qr}[1]{\ensuremath{\text{\eqref{#1}}}}
\newcommand{\ie}{i.e.\ }

\let\e=\epsilon
\newcommand{\Ordo}[1]{{O(#1)}} 
\newcommand{\ordo}[1]{{o(#1)}}

\newcommand{\lemref}[1]{Lemma~\ref{#1}}

\newcommand{\T}{\sigma}
\newcommand{\Mn}{\mathcal{M}_{\T^n}(\Sigma)}
\newcommand{\En}{\mathcal{E}_{\T^n}(\Sigma)}
\newcommand{\Mone}{\mathcal{M}_{\T}(\Sigma)}
\newcommand{\Eone}{\mathcal{E}_{\T}(\Sigma)}
\newcommand{\var}{\operatorname{var}}
\newcommand{\diam}{\operatorname{diam}}

\renewcommand{\dim}{\operatorname{dim}_H} 

\renewcommand{\d}{\,\mathrm{d}}
\newcommand{\A}{A}

\newcommand{\CF}{\mathcal{F}}
\newcommand{\CM}{\mathcal{M}}
\newcommand{\CA}{\mathcal{A}}

\newif\ifchd
\chdfalse    

\ifchd
\usepackage{tikz}
\usepackage[colorlinks=true, linkcolor=blue]{hyperref}
\usepackage{verbatim}

\makeatletter \newcommand \listoftodos{\section*{Todo list} \@starttoc{tdo}}
\newcommand\l@todo[2]
    {\par\noindent \textit{#2}, \parbox{10cm}{#1}\par} \makeatother

\newcommand{\todo}[1]{%
\addcontentsline{tdo}{todo}{\protect{#1}}%
\begin{tikzpicture}[remember picture, baseline=-0.75ex]%
    \node [coordinate] (inText) {};
\end{tikzpicture}%
\marginpar{%
    \begin{tikzpicture}[remember picture]%
        \definecolor{orange}{rgb}{1,0.5,0}
        \draw node[draw=black, fill=orange, text width = 3cm] (inNote)
                 {\small #1};%
    \end{tikzpicture}%
}%
\begin{tikzpicture}[remember picture, overlay]%
    \draw[draw = orange, thick]
        ([yshift=-0.2cm] inText)
            -| ([xshift=-0.2cm] inNote.west)
            -| (inNote.west);
\end{tikzpicture}%
}%

\let\chd=\todo
\else
\newcommand{\chd}[1]{}
\let\todo=\chd
\fi

\begin{document}

\title{Multifractal analysis of non-uniformly hyperbolic systems}
\author{A.\ Johansson, T.M.\ Jordan, A.\ \"Oberg and M.\ Pollicott}
\address{Anders Johansson \\ Divsion of Mathematics and Statistics\\
  University of G\"avle\\ SE-801 76 G\"avle\\ Sweden}
\email{ajj@hig.se}
\address{Thomas M.\ Jordan\\ Department of Mathematics\\ University of Bristol\\
  Bristol BS8 1TW, United Kingdom}
\email{thomas.jordan@bristol.ac.uk}
\address{Anders \"Oberg\\ Department of Mathematics\\ Uppsala University\\
  P.O.\ Box 480\\ SE-751 06 Uppsala\\ Sweden}
\email{anders@math.uu.se}
\address{Mark Pollicott\\ Mathematics Institute\\ University of Warwick\\
  Coventry CV4 7AL\\ United Kingdom}
\email{mpollic@maths.warwick.ac.uk}

\begin{abstract}
  We prove a multifractal formalism for Birkhoff averages of
  continuous functions in the case of some non-uniformly hyperbolic
  maps, which includes interval examples such as the
  Manneville--Pomeau map.
\end{abstract}

\maketitle

\section{Introduction and Notation}\noindent

In this paper we look at the multifractal analysis of some
non-uniformly hyperbolic maps. In particular we look at the problem
for the Birkhoff averages of continuous functions. This type of
problem is well understood in the hyperbolic case (see
\cite{BS},\cite{O},\cite{PW} for specific results and \cite{P} for an
introduction to the subject). However in the non-uniformly hyperbolic
case much less is known. The results known so far concerning Hausdorff
dimension of such spectra are limited to Lyapunov spectra
(\cite{GR},\cite{N},\cite{KS}) and local dimension of Gibbs' measures
(\cite{B}). Furthermore the methods cannot be applied to Birkhoff
averages for general continuous functions. In the case of general
continuous functions there are results for topological entropy
\cite{VerbitskyTakens}, but not for Hausdorff dimension. See also
\cite{BI} for some work on parabolic horseshoes. Finally there is work
for lcoal dimensions for countable state systems, \cite{HMU} which can
be related to parabolic systems through inducing schemes. In this
paper we produce results for the Hausdorff dimension which extend some
of the results of \cite{O} into the non-uniformly hyperbolic setting.

We begin with a classical example. Let $T:[0,1]\to [0,1]$ be the
Manneville--Pomeau map defined by $Tx = x + x^{1+\beta} \mod 1$,
where $0 < \beta < 1$. Let $f:[0,1]\to\R$ be continuous and define
\[
\Lambda_{\alpha}=
   \left\{x\in [0,1]:\lim_{n\to\infty}
     \frac{1}{n}\sum_{i=0}^{n-1}f(T^i x)= \alpha\right\}.
\] 
Let us denote, $\alpha_{\min}=\inf_{\mu\in\mathcal{E}} \{\int
f\d\mu\}$ and $\alpha_{\max}=\sup_{\mu\in\mathcal{E}}\{\int f
\d\mu\}$, where $\mathcal{E}=\mathcal{E}_T([0,1])$ denotes the space
of $T$-invariant ergodic probability measures.   

For $\alpha\in
[\alpha_{\min},\alpha_{\max}]\backslash \{f(0)\}$ we have
\[
\dim\Lambda_{\alpha}=\sup_{\mu\in\mathcal{M}_{T}([0,1])}
\left\{\frac{h(\mu,T)}{\int \log T^{\prime}(x)\d\mu}:\int
  f\d\mu=\alpha\right\},
\]
where $\mathcal{M}_T([0,1])$ denotes the $T$-invariant probability
measures. We can also show that $\dim\Lambda_{f(0)}=1.$


We can consider the related problem for iterated function
systems. Let $T_i:[0,1]\to [0,1]$, $1\leq i\leq m$, be $C^1$ maps
such that $T_i^{\prime}(x)>0$ and for distinct $i,j$, we have
$T_i(0,1)\cap T_j(0,1)=\emptyset$. At this stage the only additional
assumption we make is that $\text{diam}(T_{i_1}\circ \cdots \circ
T_{i_n}([0,1]))$ converges to $0$, uniformly in all sequences of
maps.

Let $\CA=\{1,\dots,m\}$ and let $\Sigma=\CA^\N$ be the one-sided
shift space on $m$ symbols, $\T: \Sigma\to\Sigma$ the usual shift
map and let $f:\Sigma \to {\R}$ be a continuous function.  Given $n
\geq 1$ we let $\A_n f( \w)=\frac{1}{n} \sum_{i=0}^{n-1} f(\T^i\w)$
denotes the $n$th level Birkhoff average of the function
$f:\Sigma\to\R$. Let $\Pi:\Sigma \to [0,1]$ be the natural
projection defined by
\[
\Pi(\w)=\lim_{n\to\infty} T_{\omega_1}\circ\cdots \circ
T_{\omega_n}(0), \; \; \omega \in \Sigma.
\]
Furthermore we can define the attractor of the system by
$\Lambda=\Pi(\Sigma)$. Note that by a fixed point in the iterated function system we mean the projection of a fixed point in the one-sided shift space.  We will consider the sets
\[
X_{\alpha}=\left\{\w\in \Sigma:\lim_{n\to\infty} \A_nf (\w)=
  \alpha\right\}
\]
and their images $\Pi (X_{\alpha})\subseteq\Lambda$.

Let $\tilde\Sigma$ be the subset of $\Sigma$ on which the
diameters tend to $0$ exponentially, i.e., let
\( \tilde\Sigma=\{\w: \liminf_{n\to\infty} \A_n g(\w) > 0 \}, \)
where $g(\w):= -\log\left| T'_{\omega_1}(\Pi(\w))\right|$.
We introduce the notation $\alpha_{\min}=\inf_{\mu\in\Eone} \{\int f\d\mu\}$ and
$\alpha_{\max}=\sup_{\mu\in\Eone}\{\int f \d\mu\}$. Let $\Mone$ denote the 
$\T$-invariant measures and let $\Eone$ denote the $\T$-ergodic measures. 
Denote the entropy of $\mu\in \Mone$ by $\hone\mu$ and the Lyapunov 
exponent by $\ln\mu$. We can now state our first result as follows.
\begin{theorem}\label{main}
  For $\alpha\in [\alpha_{\min},\alpha_{\max}]$ we have
\[ 
\dim \Pi (X_\alpha\cap\tilde{\Sigma})
   =\sup_{\mu\in\Mone}\left\{\frac{\hone\mu}{
    \lone\mu}:\int f\d\mu=\alpha\text{ and
  }\lone\mu>0\right\}.
\]
\end{theorem}

We now consider the specific prototype case where
$T_1,\ldots,T_m:[0,1]\to [0,1]$ are $C^1$ maps with fixed points
$(x_1,\ldots,x_m)$ such that $T_i(x_i)=x_i$, where $x_i =
\Pi(i,i,i,\dots)$, $i\in\mathcal A$. We assume that
$T^{\prime}_i(x_i)\leq 1$ and $0<T^{\prime}_i(x)<1$ everywhere else,
and for distinct $i,j$, we have $T_i(0,1)\cap
T_j(0,1)=\emptyset$. Furthermore we assume that for any $\epsilon>0$
there exists a $\mu\in \Mone$ with $\lone\mu>0$ and
$\frac{\hone\mu}{\lone\mu}\geq\dim\Lambda-\epsilon$. In other words we
have a system with a finite number of parabolic fixed points with
hyperbolic measures with dimension arbitrarily close to that of the
attractor. If the maps $T_i$ are all $C^{1+s}$ for some $s>0$ then we
can deduce from Theorem 4.6 in \cite{U} that this condition is
satisfied. This is reminiscent of Katok's result on approximation in
topological entropy by hyperbolic horseshoes, \cite{Katok}.  We can
use Theorem \ref{main} to deal with the cases where the Lypaunov
exponent is nonzero. We use a method similar to the work of Gelfert
and Rams \cite{GR} to deal with the cases where the Lyapunov exponent
can be zero. Let $\mathcal I \subset \mathcal A$ represent the set of
\emph{indifferent fixed points} so that $T'_i(x_i)=1$ whenever $i\in
\mathcal I$. For $i\in \mathcal I$, let
$\alpha_i=f(i,i,i,\ldots)$ and let
\begin{equation}\label{A}
  A=[\min_{i\in\mathcal I}\{\alpha_i\},\max_{i \in\mathcal I}\{\alpha_i\}].
\end{equation}
\begin{theorem}\label{main2}
  Assume that the iterated function system has a finite number of
  indifferent fixed points as above and that for any $\epsilon>0$ there 
exists a measure $\mu\in\Mone$ with $\lone\mu>0$ and 
$\frac{\hone\mu}{\lone\mu}\geq\dim\Lambda-\epsilon$. Then, for $\alpha\in
  [\alpha_{\text{min}},\alpha_{\text{max}}]\setminus A$ we have
\[ 
\dim \Pi(X_{\alpha})=\sup_{\mu\in\Mone}\left\{\frac{\hone\mu}{\lone\mu}:
 \int f\d\mu=\alpha\right\},
\]
and $\dim \Pi (X_{\alpha})=\dim\Lambda$ for all $\alpha\in A$.
\end{theorem}
It is straightforward to deduce that in this case $\dim\Pi
(X_{\alpha})$ is a continuous function of $\alpha$ with the possible
exception of the endpoints of $A$:
\begin{corollary}
  The function $r:[\alpha_{\min},\alpha_{\max}]\to\R$ defined by
  $r(\alpha)=\dim\Pi (X_{\alpha})$ is constant in the
  interior of $A$ and continuous in $[\alpha_{\min},\alpha_{\max}]\backslash
  A$.
\end{corollary}
\begin{proof}
  Since $r$ is clearly constant in the interior of $A$ we just
  consider $\alpha\in [\alpha_{\min},\alpha_{\max}]\setminus A$. To
  start let $\mu_1,\mu_2\in\Mone$ such that $\int
  f\d\mu_1=\alpha_{\min}$ and $\int f\d\mu_2=\alpha_{\max}$. For
  $\alpha\in[\alpha_{\min},\alpha_{\max}]\backslash A$ consider a
  sequence $\{\beta_n\}_{n\in\N}$ such that $\beta_n\rightarrow\alpha$
  as $n\rightarrow\infty$. It follows that
  $r(\alpha)\geq\limsup_{n\rightarrow\infty}r(\beta_n)$ by upper
  semicontinuity of entropy (see Theorem 8.2 \cite{W}). For
  $\epsilon>0$ we let $\mu$ satisfy $\int f\d\mu=\alpha$ and
  $\frac{\hone\mu}{\lone\mu}>r(\alpha)-\epsilon$. By considering
  measures $\nu_n$ such that $\int f\d\nu_n = \beta_n$ of the form
  $\nu_n=p_n\mu_1+(1-p_n)\mu$ or $\nu_n=p_n\mu_2+(1-p_n)\mu$ for
  appropriate $p_n\!\!\searrow\!0$, we can deduce that
  $r(\alpha)\leq\liminf_{n\rightarrow\infty} r(\beta_n)$. 
\end{proof}
These results are well understood in the case of uniformly
contracting systems (see \cite{PW},\cite{O},\cite{BS}). The novelty
in this work is that we can analyse certain non-uniformly hyperbolic
systems. Moreover, our methods do not involve either thermodynamic
formalism or the use of large deviation theory. For an introduction
to dimension theory and multifractal analysis the reader is referred
to \cite{P}. All the necessary definitions and results from ergodic
theory can be found in \cite{W}.

We can also use Theorem \ref{main2} to deduce a result which applies
to non-uniformly expanding maps of the interval (such as the
Manneville--Pomeau map mentioned earlier).
\begin{corollary}\label{nue}
  Let $T:[0,1]\to [0,1]$ be a piecewise onto $C^1$ map
  with a finite number of parabolic fixed points $x_i$ such that
  $T(x_i)=x_i$ and $T'(x_i)=1$ but $T'(x)>1$ for $x\in
  [0,1]\backslash\cup_i x_i$. We also assume the existence of a 
hyperbolic measure with dimension arbitrarily close to $1$. Let 
$f:[0,1]\to\R$ be continuous and let
  \[
  \Lambda_{\alpha}=\{x\in [0,1]:\lim_{n\to\infty}
  \frac{1}{n}\sum_{i=0}^{n-1}f(T^i x)= \alpha\}.
  \]
  If we let $A=[\min_i \{f(x_i)\},\max_i\{f(x_i)\}]$ then for
  $\alpha\in [\alpha_{\min},\alpha_{\max}]\backslash A$ we have
  \[
  \dim\Lambda_{\alpha}=\sup_{\mu\in\mathcal{M}_{T}([0,1])}
  \left\{\frac{h(\mu,T)}{\int \log T^{\prime}(x)\d\mu(x)}:\int
    f\d\mu=\alpha\right\}.
  \]
  We also have that for all $\alpha\in A$ 
  \[\dim\Lambda_{\alpha}=1.\]
\end{corollary}

\begin{proof}
  This follows by noting that Theorem \ref{main2} can be applied to
  the iterated function system defined by the inverse branches of
  this map.
\end{proof}
Without the assumption of the existence of a hyperbolic measure of 
dimension arbitrarily close to $1$ the result would be the same 
except that we would no longer have equality for $\alpha\in A$ but 
that the dimension is bigger than the supremum of the dimension of 
hyperbolic measures and less than the dimension of the attractor. 
We don't know of any examples where this situation occurs. It is also possible to generalise the result to Markov maps however we just work in the Bernoulli case to ease the exposition. It follows 
from Theorem 4.6 in \cite{U} that for any such system where the inverse 
branches are $C^{1+s}$ for $s>0$ that this hypothesis is satisfied. 
We now give some examples to illustrate this corollary and the
difference of the result from the expanding case.

\begin{example}
  The Manneville--Pomeau map is known to have a finite $T$-invariant
  absolutely continuous probability measure (we denote this measure by
  $\mu_{SRB}$) in the case when $0<\beta<1$. For $\beta\geq 1$ there 
is no $T$-invariant absolutely continuous probability measure but 
there are measures of dimension arbitrarily close to $1$. 
    So  provided $\beta >0$ it  satisfies the hypotheses
  of the Corollary \ref{nue}.  Let $f:[0,1]\to\R$ be a continuous
  function such that $\int f \d\mu_{SRB}>f(0)=\alpha_{\min}$.  Then
  for $\beta \in (0,1)$ we have that
  \[
  \dim\Lambda_{\alpha}\left\{\begin{array}{lll}=1&\text{ for
      }&\alpha\in
      \left[f(0),\int f\d\mu_{SRB}\right]\\
      <1&\text{ for }&\alpha\in\left[\int
        f\d\mu_{SRB},\alpha_{\max}\right]
    \end{array}\right.
  \]
  In the case where $\beta > 1$ we have that $\dim\Lambda_{f(0)}=1$
  and $\dim\Lambda_{\alpha}<1$ for $\alpha\in(f(0),\alpha_{\max})$.
\end{example}
\begin{example}
  Let $T:[0,1]\rightarrow [0,1]$ be defined by
  \[
  T(x)=\left\{\begin{array}{lll}\frac{x}{1-x}&\text{ for }&0\leq
      x\leq\frac{1}{2}\\\frac{2x-1}{x}&\text{ for
      }&\frac{1}{2}<x\leq 1\end{array}.\right.
  \]
  Thus $T$ has parabolic fixed points at $0$ and $1$ but is
  expanding everywhere else. There are no absolutely continuous 
$T$-invariant probability measures but there are $T$-invariant measures 
of dimension arbitrarily close to $1$.  Hence for any $f:[0,1]\rightarrow\R$
  which is continuous we can apply Corollary \ref{nue}. In the case
  where $f$ is monotone increasing then
  $A=[f(0),f(1)]=[\alpha_{\min},\alpha_{\max}]$ and $\dim
  \Lambda_{\alpha}=1$ for all $\alpha\in
  [\alpha_{\min},\alpha_{\max}]$.
\end{example}
The layout of the rest of the paper is as follows. In the next three
sections we give the proof of Theorem \ref{main}. In the final
section we go on to deduce Theorem \ref{main2}.

\section{Preliminary Results}
In this section we prove the basic lemmas needed to prove Theorem
\ref{main}. These involve an approximation result and methods of
invariant and ergodic invariant measures. We will use $\Mn$, and $\En$ to
denote the $\T^n$-invariant and $\T^n$-ergodic measures respectively. For
$\mu\in\Mn$ let $\hn{\mu}$ denote the entropy of $\mu$ with
respect to $\T^n$. Let $\CF_n$ be the finite algebra generated by
the $n$-cylinders. Since, for each $n$, $\CF_n$ is a generating
partition for $\T^{n}$, we have for $\mu\in\Mn$ that $\hn{\mu} =
n \hone \mu$ where
\[
\hone\mu=\lim_{N\to\infty}\frac 1N \,H\left(\mu\vert_{\CF_N}\right)
\]
and $H(\nu)$ denotes the Shannon entropy
\[ H(\nu) = \sum_{A \in \CA} \nu(A)\log(1/\nu(A)), \] defined for
probabilities $\nu$ on finite algebras $\CA$.

Let $g:\Sigma\to {\R}$ be defined by $g(\w)=-\log
T_{\w_1}^{\prime}(\Pi(\T(\w)))$.  We then denote the Lyapunov
exponent of a measure $\mu\in\Mn$ by
$$\ln\mu= \int \sum_{k=0}^{n-1}g(\T^k\w)\d\mu(\w).$$

Let $\pi_n:\Sigma\to\CA^{n}$ denote the natural
projection onto the first $n$ symbols, \ie $\pi_n(\w) \mapsto
(\w_1,\dots,\w_n) \in \CA^n$. For cylinder sets we use the notation
$[\alpha_1,\ldots,\alpha_n]$ for
$\pi_n^{-1}(\alpha_1,\ldots,\alpha_n)$ and for $\w\in\Sigma$ we let
$[\w]_n$ denote $\pi_n^{-1}(\pi_n(\w))$. Let $\CF_n$ be the finite 
$\sigma$-algebra generated by the $n$-cylinders $\{[\w]_n: \w\in\Sigma\}$.

For a function $f:\Sigma\to \R$, define the $n$th variation as
\[ \var_n f = \sup_{[\w']_n=[\w]_n} |f(\w)-f(\w')|. \] By definition,
$\displaystyle\lim_{n\rightarrow\infty} \var_n f = 0$ if $f$ is
continuous and then also 
$\displaystyle\lim_{n\rightarrow\infty} \var_n \A_n f = 0$.

Let $I_n(\w)\subset I$ denote the interval $T_{\w_1}\circ\cdots\circ
T_{\w_n}([0,1])$, where $I_0(\w)\equiv [0,1]$, and let the
corresponding diameters be given by $D_n(\w)=\diam(I_n(\w))$. For
$n\geq 1$, we write $\ll_n(\w)$ for $-\frac 1n\,\log \D_n(\w)$.

We start by showing that, for lage $n$,  $\ll_n$ is well approximated
by the Birkhoff average $\A_n g(\w)$.
\begin{lemma}\label{conv}
Let $T_i:[0,1]\to [0,1]$, $1\leq i\leq m$, be $C^1$ maps
such that $T_i^{\prime}(x)>0$ and such that for distinct $i,j$, we have  
$T_i(0,1)\cap T_j(0,1)=\emptyset$. In addition we assume that 
$D_n(\w)\to 0$, uniformly in $\w$. Then
  \[ \lim_{n\rightarrow\infty}\sup_{\omega\in\Sigma}\{|-\frac 1n\, 
\log\D_n(\w) - \A_n g(\w)|\}= 0
  \]
\end{lemma}
\begin{proof}
  We introduce the functions $g_n:\Sigma\to \R$, $n \geq
  1$, defined by
  \[
  g_n(\w)= -\log\frac{D_n(\w)}{D_{n-1}(\T\w)}.
  \]
  It is immediate from the definitions that
  \[
  -\log D_n(\w) = \sum_{i=0}^{n-1} g_{n-i}(\T^i\w).
  \]
  We can relate this identity to the Birkhoff averages of
  $g:\Sigma\to\R$ using the following fact
  \begin{equation}\label{uniform}
    g_n(\w) \to g(\w)\quad\text{uniformly in $\w$ as $n\to\infty$}. 
  \end{equation}
  To see (\ref{uniform}), we note that for $n \geq 2$,
  \[
  g_n(\w)=-\log \left(\frac{1}{D_{n-1}(\T\w)}\int_{I_{n-1}(\T\w)}
    T_{\w_1}'(x) \; dx \right) = -\log T_{\w_1}'(\xi)
  \]
  for some $\xi\in I_{n-1}(\T\w)$ by the intermediate value
  theorem. By hypothesis, each $\log T_i'(x)$ is (uniformly)
  continuous and thus, since $\diam(I_{n-1}(\T\w))=D_{n-1}(\T\w)$
  tends to $0$ uniformly, it follows that
  \[ 
    g_n(\w)-g(\w)= \log T_{\w_1}'(\Pi(\T\w))-\log T_{\w_1}'(\xi) 
  \]
  also tends to $0$ uniformly as $n\to\infty$.

   Let
  $\e>0$ and note that by \qr{uniform}, we can choose $N_1$ such 
that for $n\geq N_1$ and $\w\in\Sigma$ we have $|g_n(\w)-g(\w)|
\leq\frac{\e}{2}$. We can also find a $C_1>0$ where $|g_n(\w)-g(\w)|
  < C_1$ for all $\omega\in\Sigma$.  Let $N=\left(\left[\frac{2C_1}{\e} 
\right] + 1\right)N_1$.

  For $n\geq N$ and $\w\in\Sigma$ we have that
  \begin{align*}
    -\log\D_n(\w)&=\sum_{i=0}^{n-1} g_{n-i}(\T^i\w)\\
    &\leq \sum_{i=0}^{n-1-N_1}\left(g(\T^{i}\w\right)+
           \frac{\e}{2})+\sum_{i=n-N_1}^{n-1}g_{n-i}(\T^i\w)\\
    &\leq \sum_{i=0}^{n-1}g(\T^{i}\w)+(n-N_1)\frac{\epsilon}{2}+C_1N_1\\
    &\leq \sum_{i=0}^{n-1}g(\T^{i}\w)+(n-N)\frac{\epsilon}{2} + N \frac{\e}{2}\text{ (since $C_1N_1\leq \frac{N\epsilon}{2}$)}\\
    &\leq \sum_{i=0}^{n-1}\left(g(\T^i(\w))+\e\right).
  \end{align*}
  The other inequality is similar.
\end{proof}

We now need to relate $\T^n$-ergodic measures to $\T$-ergodic
measures and $\T$-invariant measures to $\T^n$-ergodic measures.
Given $\nu\in \En$ we define $\mu = \A_n^*\nu$ as the measure
\[ \mu = \frac 1n \, \sum_{k=0}^{n-1} \nu\circ \T^{-k}. \]
\begin{lemma}\label{decyclo}
   If $\mu=\A_n^*\nu$, $\nu\in\En$, then $\mu\in\Eone$ and 
   \begin{enumerate}
   \item $\hone{\mu}=\frac{1}{n} \hn\nu.$
   \item $\lone\mu =\frac{1}{n}\ln\nu$.
   \item $\int f\d\mu=\int \A_nf \d\nu$.
   \end{enumerate}
 \end{lemma}
 \begin{proof}
   The first part of this lemma follows by Abramov's
   Theorem (see \cite{W} Theorem 4.13). The final two parts are routine calculations.
 \end{proof}

 \newcommand{\otim}[1][]{^{\otimes}}

 The next lemma states that we may approximate any invariant measure
 $\mu \in \Mone$ by ergodic measures in $\En$.  A probability
 measure $\mu$ on $\Sigma$ is \emph{$n$th level Bernoulli} if the
 \emph{$n$-blocks} $\pi_n\circ T^{kn}(\w)$ are independent and
 identically distributed for $k=0,1,\dots$. An $n$th level Bernoulli
 measure is always $\T^n$-invariant and ergodic with respect to
 $\T^n$.  Moreover, we have a natural continuous bijection
 $\nu\mapsto\nu\otim$ between probabilities on blocks
 $\nu\in\CM(\CA^n)$ and the corresponding $n$th level Bernoulli
 measures. The block probability $\nu$ is the marginal of the
 corresponding $n$th level Bernoulli measure, \ie
 $(\nu\otim)\circ\pi_n^{-1} = \nu$, and we have
 $\hn {\nu\otim} = H(\nu)$.

 \begin{lemma}\label{approxim}
   For any $\mu\in\Mone$, 
   we can find a sequence of
   measures $\{\mu_n\}$ converging to $\mu$ in the weak$^{\star}$-topology
   such that
   \begin{enumerate}
   \item\label{xx1} $\mu_n$ is $n$th level Bernoulli,
   \item\label{xx3} $\lim_{n\to\infty} \frac1n {\hn{\mu_n}} =
     \hone\mu$ and
     $\lim_{n\to\infty}\frac1n \ln{\mu_n}=\lone\mu$,
   \end{enumerate}
   and moreover, if $\int f\d\mu=\alpha\in
   (\alpha_{\min}, \alpha_{\max})$, 
   then we may in addition assume that
   \begin{enumerate}
   \item[(3)]\label{xx4} $\int \A_nf\d\mu_n= \alpha$.
   \end{enumerate}
 \end{lemma}

 \begin{proof}[Proof of Lemma \ref{approxim}]
   To see the first part, let $\mu_n =
   (\mu\circ\pi_n^{-1})\otim$. Then $\mu_n\vert_{\CF_n} =
   \mu\vert_{\CF_n}$ and, since $\CF_n$ increases to the Borel
   $\sigma$-algebra, this implies that $\mu_n\to\mu$ in the 
weak$^{\star}$-topology. Also,
   \[
   \lone\mu = \lim_{n\to\infty} \int \A_n g \d\mu_n \text{ and }
   \alpha=\int f\d\mu = \lim_{n\to\infty} \int \A_n f \d\mu_n
   \]
   and by definition we have
   \[ \hone \mu =\lim_{n\to\infty} \frac{1}{n}
   H(\mu\circ\pi^{-1}_n) = \lim_{n \to +\infty } \frac{1}{n}
   \hn {\mu_n}. \]
   
   We need to work a little bit more to modify this construction to
   give a sequence $\tilde\mu_n$ of $n$th level Bernoulli measures
   that also satisfies \eqref{xx4}, in addition to \eqref{xx1} and
   \eqref{xx3}.  Without loss of generality, we may assume that
   $\int \A_{n_j}f\d\mu_{n_j} \leq \alpha$ for some infinite
   sequence $\mathcal N = \{n_j\}$. We proceed to construct
   $\tilde\mu_n$ for such $n$'s and a symmetric construction gives
   $\tilde\mu_n$ for $n\in\N \setminus \mathcal N$.

   By the ergodic theorem, we can always find a point $x\in\Sigma$,
   a number $0<\rho<(\alpha_{max}-\alpha)/3$ and an integer $N>0$
   such that $\A_nf(x)\geq \alpha + 3\rho$ for all $n \geq N$.
   Denote by $\nu_n = \mu_n\circ\pi_n^{-1} \in \CM(\CA^n)$ the block
   marginal of $\mu_n$ and let $\delta_{n} \in\CM(\CA^n)$ denote the
   Dirac measure at the word $(x_1,x_2,\dots,x_n)$.  Then
   $\delta_n\otim$ is a Dirac measure on the corresponding periodic
   sequence, and we can also assume that $\int \A_n f
   \d(\delta_n\otim) > \alpha + 2\rho$ for all $n\geq N$.

   Furthermore, define for $ s \in [0,1]$ a $n$th level
   Bernoulli measure
   \[ \xi_{s,n}:=(s \nu_n +
   (1-s)\delta_n)\otim. \] Note that the map \( s\mapsto
   F(s)=\int \A_n f \d\xi_{s,n} \) is continuous and
   hence, since $n\in \mathcal N$ means that $F(1)\leq \alpha$ and
   $F(0)>\alpha+2\rho$, we deduce that \( \int \A_n f
   \d\xi_{s_n,n} = \alpha, \) for some $s_n\in[0,1]$.

   We need to show that we can choose $s_n\to 1$ as $n$ tends
   to infinity with $n\in\mathcal N$, since then by setting
   $\tilde\mu_n = \xi_{s,n}$ we have
   \[ \frac{1}{n} \hn {\tilde\mu_n} = s_n \frac{1}{n}
   \hn{\mu_n} + (1-s_n) \cdot 0 + \Ordo{1/n} \to
   \hone{\mu} \] and
   \[
     \lone{\tilde\mu_n} = s_n \cdot \lone{\mu_n} +
     (1-s_n) \cdot \int \A_n g\,\d(\delta_n\otim) + 
     \Ordo{\var_n \A_n g} \to \lone{\mu} 
   \]
   as required. The fact that $\lim_{n\rightarrow\infty}\var_n A_ng=0$ follows from the uniform continuity of $g$.

   For all $s \in [0,1]$, the $n$-block marginals for
   $\xi_{s,n}$ and
   \[
   \zeta_{s,n}: =s \nu_n\otim + (1-s)\delta_n\otim
   \]
   coincide, \ie $\xi_{s,n} \circ \pi_n^{-1}=\zeta_{s,n}
   \circ \pi_n^{-1}$. Therefore
   \[
   \int \A_n f \d\xi_{s,n} \geq \int\A_n f
   \d\zeta_{s,n}-\var_n\A_n f.
   \]
   and since $\nu_n\otim = \mu_n$ we deduce
   \begin{align*}
     \int\A_n f \d\zeta_{s,n} &= s\int\A_n f \d\mu_n +
     (1-s) \,  \int \A_n f \d(\delta_n\otim) \\
     &\geq \alpha - s\epsilon_{n} + (1-s)\cdot (2\rho),
   \end{align*}
   where $\epsilon_n = \alpha - \int \A_n f \d\mu_n$.
   
   Thus
   \begin{align*}
     \alpha &= \int\A_n f \d{\tilde\mu_n} \\
     &\geq \alpha - s\epsilon_{n} + (1-s)\cdot (2\rho) -
     \var_n\A_n f,
   \end{align*}
   which, since $\e_n\to 0$ and $\var_n\A_n f\to0$ as $n\in\mathcal
   N$ tends to infinity, gives that $\lim_{n \to +\infty} s_{n} =
   1$.
 \end{proof}

\section{Lower Bound}\label{lowerbound}
In this section we shall prove the lower bound in
Theorem~\ref{main}, \ie 
\begin{equation}\label{lbnd}
  \dim\Pi(X_{\alpha} \cap \tilde\Sigma ) \geq
  \sup_{\mu\in\Mone} \left\{
   \frac{\hone\mu}{\lone\mu}:
   \int f\d\mu=\alpha,\lone\mu>0 
  \right\}.
\end{equation}
We start by calculating the dimension of the projection of any
invariant measure $\mu\in\En$ with positive Lypaunov exponent.
\begin{lemma}[Hofbauer--Raith]\label{hr}
  Let $\mu\in\Mn$ be ergodic with respect to $\T^n$ and satisfy
  $\ln\mu>0$. We have that
\[ \dim(\mu\circ\Pi^{-1}) =\frac{\hn{\mu}}{\ln\mu}. \]
\end{lemma}
\begin{proof}
  This was originally shown by Hofbauer and Raith in \cite{HR}. The
  proof can be seen by applying Lemma \ref{conv}, together with the
  Birkhoff Ergodic Theorem and the Shannon--McMillan--Breiman
  Theorem.
\end{proof}

Let $\mu\in\En$ satisfy both $\int \A_n f\d\mu=\alpha$ and
$\ln\mu>0$.  
It follows from the Birkhoff Ergodic
Theorem that for any such $\mu$ we have
$\mu(X_\alpha\cap\tilde{\Sigma})=1$.  Hence we can deduce from 
Lemma \ref{hr} that
$$\dim\Pi (X_{\alpha}\cap\tilde{\Sigma})\geq\frac{\hn\mu}{\ln\mu}$$
and so 
\[ 
\dim\Pi (X_{\alpha}\cap\tilde{\Sigma})\geq 
  \sup_{\mu\in\En}
  \left\{
   \frac{\hn\mu}{\ln\mu}: 
   \int \A_n f\d\mu=\alpha\text{ and }\ln\mu>0 
  \right\}. 
\]
We
can now apply Lemma \ref{approxim} to see that
\[
\dim\Pi (X_{\alpha}\cap\tilde \Sigma)\geq
\sup_{\mu\in\Mone}\left\{\frac{\hone\mu}{\lone\mu}:\int
  f\d\mu=\alpha\text{ and }\lone\mu>0\right\}
\]
for $\alpha\in (\alpha_{\min},\alpha_{\max}).$

The cases when $\alpha=\alpha_{\min}$ or $\alpha=\alpha_{\max}$ need
to be handled separately since we cannot apply Lemma
\ref{approxim}. However it can be seen from the ergodic
decomposition of such an invariant measure that
\begin{displaymath}
  \sup\left\{\frac{\hone\mu}{\lone\mu}: 
      \mu\in\Mone,\ \int f\d\mu=\alpha_{\min}\right\}
\end{displaymath}
must be the same as
\begin{displaymath}
  \sup\left\{\frac{\hone\mu}{\lone\mu}:\mu\in\Eone, \int 
f\d\mu=\alpha_{\min}\right\}.
\end{displaymath}
If $\mu'\in\Eone$ occurs in an ergodic
decomposition of $\mu\in\Mone$ where $\int f\d\mu = \alpha_{min}$,
then, by the extremality of $\alpha_{min}$,  $\int f\d\mu' =
\alpha_{min}$. 
The same is true for $\alpha_{\max}$. This completes the proof of
the lower bound.

\section{Upper Bound}

In this section we shall prove the upper bound in
Theorem~\ref{main}, \ie 
\begin{equation}\label{ubnd}
  \dim \Pi (X_{\alpha} \cap \tilde\Sigma) \leq
  \sup_{\mu\in\Mone} \left\{
   \frac{\hone\mu}{\lone\mu}:
   \int f\d\mu=\alpha,\lone\mu>0 
  \right\}.
\end{equation}

For $\delta > 0$ let
\[
\tilde{\Sigma}(\delta) =
\{\omega\in\Sigma:\liminf_{n\rightarrow\infty}A_ng(\omega)\geq \delta\}.
\]
so that $\tilde{\Sigma}$ can be written as the countable union
$\tilde{\Sigma}=\cup_j \tilde{\Sigma}(\delta_j)$, for any
sequence $\{\delta_j\}$ with $\displaystyle\lim_{j\rightarrow\infty}\delta_j = 0$. 

Recall that
\[
\D_n(\w)= \diam\left(T_{\w_1}\circ\cdots\circ T_{\w_n}([0,1])\right)
\text{ and } \ll_n(\w)= -\frac 1n \, \log \D_n(\w). 
\]
An important
consequence of Lemma~\ref{conv} is that for any $\eta>0$
there exists some $N_0=N_0(\eta)$ such that
\begin{equation}\label{approx}
  \ll_n(\w)  \geq (1-\eta) \A_n g(\w)
\end{equation}
for all $n\geq N_0$ and all $\w\in \tilde{\Sigma}(\delta)$. (To see
this, take $\epsilon = \eta\delta$ in the proof of Lemma~\ref{conv}.)

Most of the section will be devoted to proving the following lemma.
We consider the sets
\begin{equation*}
  X(\alpha,N,\rho,\delta) =
  \left\{\w\in\tilde{\Sigma}(\delta): 
  \A_nf(\w)\in B(\alpha,\rho)\text{ for all }
  n\geq N\right\},
\end{equation*}
where, for $\rho>0$, $B(\alpha,\rho)=\{x:|x-\alpha|<\rho\}$.
\begin{lemma}\label{notso}
  For all $\rho,\epsilon>0$, $\delta>0$ sufficiently small and $N\in\N$ 
we can find a measure
  $\mu\in\Mone$ such that $\int f\d\mu\in B(\alpha,2\rho)$,
  $\lone\mu>\delta$ and
  \[ \dim \Pi (X(\alpha,N,\rho,\delta)) \leq
  \frac{\hone\mu}{\lone\mu}+\epsilon. \]
\end{lemma}
Before proving this lemma we show how it can be used to deduce
\eqref{ubnd}. 
\begin{proof}[Proof of \eqref{ubnd}]
  First note that, for any fixed $\rho>0$ and $\delta>0$, the set
  $X_{\alpha}\cap\tilde{\Sigma}(\delta)$ is contained in the
  increasing union $\cup_{N\in\N} X(\alpha,N,\rho,\delta$ and thus
  \[
  \dim \Pi (X_\alpha \cap \tilde{\Sigma}(\delta)) 
  \leq \sup_{N} \dim\Pi (X(\alpha,N,\rho,\delta)).
  \]
  For any $\epsilon,\delta>0$ and any sequence $\rho_n$, decreasing
  to zero as $n\to\infty$, we can use \lemref{notso} to find a
  sequence of invariant measures $\mu_n\in\Mone$ with $\int
  f\d\mu_n\in B(\alpha,2\rho_n)$, $\lone{\mu_n}>\delta$ and
  \[
  \dim \Pi (X_\alpha)\leq \dim\Pi (X(\alpha,N_n,\rho_n, \delta)) +
  \epsilon/2 \leq \frac{\hone{\mu_n}}{\lone{\mu_n}}+\epsilon.
  \]

  If $\mu_{\delta}$ is any weak$^{\star}$-limit of $\mu_n$ it
  clearly follows that $\int f\d\mu_{\delta}=\alpha$, and from the
  upper semicontinuity of entropy  $\mu\to \hone{\mu}$ (see \cite{W} Theorem 8.2)
  and the continuity of $\mu\to\lone\mu$ it follows that
  \[
  \dim\Pi
  (X_{\alpha}\cap\tilde{\Sigma}(\delta))\leq
  \frac{\hone{\mu_{\delta}}}{\lone{\mu_{\delta}}}+\epsilon,
  \]
  where $\lone{\mu_{\delta}}>\delta$. This holds for any
  $\delta>0$ so by taking a countable union of $\delta_n$ where
  $\delta_n\rightarrow 0$ we get
  \[
  \dim\Pi
  (X_{\alpha}\cap\tilde{\Sigma})\leq
  \sup_{n}\left\{
     \frac{\hone{\mu_{\delta_n}}}{\lone{\mu_{\delta_n}}}+\epsilon
  \right\}.
  \]
  Since $\epsilon$ was arbitrary, this completes the proof of
  \eqref{ubnd}. 
\end{proof}

\subsection*{Proof of Lemma \ref{notso}}
Fix $\epsilon>0$, $N\in \N$, $\alpha,\delta,\rho>0$ and let
$X=X(\alpha,N,\rho,\delta)$.  By compactness, $f$ and $g$ are
uniformly continuous and if $N$ is sufficiently large then for all
$n\geq N$,
\begin{equation}\label{Xineq}
  |\A_nf(\tau)-\A_nf(\w)|\leq \rho/2 
\end{equation}
whenever $[\w]_n = [\tau]_n$. Furthermore, by 
\eqref{approx} we can assume that for all $n\geq N$ and for all $\w \in X$,
\begin{equation}\label{approx2}
  \A_n g(\w) \leq (1+\epsilon) \ll_n(\w). 
\end{equation}

For $n\geq N$, let $Y_n$ consist of all cylinders $[\w]_n \in \CF_n$
which contain a point in $X$, \ie $Y_n = \pi_n(X)$. For each $n$,
define $s_n$ to be the solution to
\begin{equation}\label{sndef}
  \sum_{[\w]_n\in Y_n} D_n^{s_n}(\w)=1.
\end{equation}
From the definition of Hausdorff dimension, it then immediately
follows that 
$$\dim\Pi (X) \leq \liminf_{n\to\infty} s_n,$$
since the projections of the elements of $Y_n$ form a sequence of covers 
of $\Pi (X)$ by intervals having diameters decreasing to zero as $n\to\infty$.

Let $\nu_n$ be the probability defined by \eqref{sndef} on the
$n$-cylinders, \ie
\[ \nu_n([\w]_n) = \D_n^{s_n}(\w) = e^{-ns_n \ll_n(\w)} \] if
$[\w]\in Y_n$ and zero otherwise. Let $\mu_n$ be the corresponding
$n$th level Bernoulli measure defined by
$\nu_n=\mu_n\vert_{\CF_n}$. That is, $\mu_n = \nu_n\otim$ where
$\nu_n$ is interpreted as a measure in $\CM(\CA^n)$.  It is clear
from \eqref{Xineq} that $\int \A_n f\d\mu_n\in{}B(\alpha,2\rho)$,
since each cylinder $[\w]_n\in Y_n$ only contains $\tau\in\Sigma$
such that \( \A_nf (\tau)\in B(\alpha,2\rho)\).

Evaluating the Shannon entropy $H(\nu_n)$ of $\nu_n$ gives the
equality
\begin{equation*}
  H(\nu_n) = n s_n \int \ll_n \d\nu_n,
\end{equation*}
where the integral denotes the expected value
$\sum_{[\w]_n}\nu([\w]_n) \, \ll_n(\w)$ of $\ll_n$ with respect to
$\nu_n$.  Since $\mu_n=\nu_n\otim$ and $\ll_n$ is
$\CF_n$-measurable, it is easy to see that $H(\nu_n) =
\hone{\mu_n}$
and that $\int\ll_n\d\mu_n=\int\ll_n\d\nu_n$.

Thus we have the identity
\begin{equation*}
  \hone{\mu_n}= s_n \int \ll_n \d\mu_n.
\end{equation*}

In view of \eqref{approx2}, this means that
\begin{equation}\label{Ll}
  \hn {\mu_n} \geq s_n (1-\epsilon) \int \A_ng \d\mu_n 
  = (1-\epsilon) s_n \ln{\mu_n}
\end{equation}
since $\mu_n$ is $n$th level Bernoulli and $\T^n$-invariant. Thus
\begin{equation}
  s_n \leq \frac{1}{(1-\e)}\, \frac{\hn{\mu_n}}{\ln{\mu_n}}
\end{equation}

To complete the proof of Lemma \ref{notso}, simply note that Lemma
\ref{decyclo} implies that $\A_n^* \mu_n$ belongs to $\Eone$ and
that $\hone{A^*_n\mu_n} = \frac1n \hn{\mu_n}$ and that
$\lone{A^*_n \mu_n}=\frac1n \ln{\mu_n}$. Moreover,
$\int \A_nf\d\nu_n\in B(\alpha,2\rho)$.  \qed

\section{Proof of Theorem \ref{main2}}

We now need to consider the sequences $\w$ where
$\liminf_{n\to\infty} \A_ng(\w)=0$ and we cannot apply Theorem
\ref{main}. We start by showing that if the limit of the Birkhoff
average for such a sequence exists it must necessarily lie in $A$
(see (\ref{A}) for the definition).
\begin{lemma}\label{zl}
  Let $\{n_j\}_{j\in\N}$ be a subsequence of $\N$. If for any
  $\w\in\Sigma$, $\lim_{j\to\infty} \A_{n_j}$$g(\w)=0$, then we
  have that $\lim_{j\to\infty}\A_{n_j}f(\w)\in A$ if the limit
  exists. In particular, this shows that
  $\lim_{n\to\infty}\A_ng(\w)=0$ means
  $\lim_{n\to\infty}\A_nf(\w)\in A$, if the limit exists.
\end{lemma}
\begin{proof}
  Let $\epsilon,l>0$. We can find $\delta>0$ such that $g(\w)<\delta$
  implies that $\text{dist}(f(\omega),A)<\frac{\epsilon}{2}$. If
  $\frac{1}{n_j} \A_{n_j}g(\w)<\frac{\delta}{l}$ then
  $g(\T^i\w)<\delta$ for at least $\frac{n_j(l-1)}{l}$ of the values
  $1\leq i\leq n_j$. Thus $\text{dist}(\A_{n_j}f(\w),A)\leq
  \frac{(l-1)}{l}\epsilon+\frac{K}{l}$, where
  $K=\sup_\w \operatorname{dist}(f(\w),A)$. Since $l$ is arbitrary
  this completes the proof.
\end{proof}
Hence we know that when $\alpha\notin A$, we do not need to consider 
those $\omega$ which satisfy $\liminf_{n\rightarrow\infty}A_ng(\omega)=0$, 
since then, if $\lim_{n\rightarrow\infty} A_nf(\omega)$ exists, it can only 
take the values in $A$. Thus the first part of Theorem \ref{main2} follows 
immediately from Theorem \ref{main}.

By renaming the symbols, we may arrange so that 
\[ A=[a_1,a_2] = [f(\w_1), f(\w_2) ], \]
where $\w_1=(1,1,1,\dots)$ and $\w_2=(2,2,2,\dots)$ correspond to the
the indifferent fixed points $x_1 =\Pi(\w_1)$ and $x_2=\Pi(\w_2)$,
respectively.  
We denote by $\delta_{1}$
and $\delta_{2}$ the Dirac measures on the sequences
$(1,1,1,\dots)$ and $(2,2,2,\dots)$.

We will need to consider separately the two cases where $\alpha\in
(a_1,a_2)$ and when
$\alpha\in\{a_1,a_2\}$. For the first case we need
the following lemma.
\begin{lemma}\label{interior}
  For any $\alpha\in (a_1,a_2)$ and $\nu\in\Eone$ such that
  $\lone{\nu}>0$ we can find $\mu\in\Mone$ such that $\int
  f\d\mu=\alpha$ and
$$\frac{\hone\mu}{\lone\mu}=\frac{\hone\nu}{\lone\nu}.$$
\end{lemma}
\begin{proof}
  We assume without loss
  of generality that $\int f\d\nu\geq\alpha$. We can then find
  $p_1,p_2>0$ such that $p_1+p_2=1$, $p_2>0$ and $p_1 a_1+p_2\int
  f\d\nu=\alpha$. We then let $\mu=p_1\delta_{1}+p_2\nu$.
  and deduce that
$$\int f\d\mu=p_1 a_1+p_2\int f\d\mu=\alpha,$$
and
$$\frac{\hone\mu}{\lone\mu}=\frac{p_2 \hone\nu}{p_2\lone\nu}
  =\frac{\hone\nu}{\lone\nu}.$$
\end{proof}

For $\alpha\in (a_1,a_2)$ and $\epsilon>0$ we take
$\nu\in \Eone$ such that $\frac{\hone\nu}{\lone\nu}\geq
\dim\Lambda-\epsilon$. By applying Lemma \ref{interior} we can find a
measure $\mu\in\Mone$ such that $\int f\d\mu=\alpha$ and
$\frac{\hone\mu}{\lone\mu}\geq\dim\Lambda-2\epsilon$.  By combining
Lemmas \ref{decyclo} and \ref{approxim} we can find a measure
$\eta\in\Eone$ such that $\int f\d\eta=\alpha$ and
$\frac{\hone\eta}{\lone\eta}\geq \dim\Lambda-2\epsilon$. It thus
follows that $\eta(\Lambda_{\alpha})=1$ and by using Lemma \ref{hr} we
can see that
$$\dim\Lambda_{\alpha}\geq\dim\Lambda-2\epsilon.$$

To complete the proof for the second case where
$\alpha\in\{a_1,a_2\}$ we follow a similar approach
to that used by Gelfert and Rams in \cite{GR}.

Our strategy is to look at sequences which alternate between a
hyperbolic measure of large dimension and the parabolic measure at the
fixed point. We arrange that they spend more time at the fixed point
and so this will determine the Birkhoff average. However, if the
proportion of time at the fixed point does not grow to quickly with
relation to the proportion of time described by the hyperbolic measure, 
then the dimension can be given by the hyperbolic measure.

Without loss of generality, we may assume that $\alpha=a_1$.
With the notation introduced before Lemma~\ref{interior}, we study the
behaviour about the indifferent fixed point $x_1=\Pi(\w_1)$, where
$\w_1=(1,1,1,\dots)$ and where $\int f\d\delta_{1}=\alpha$ and
$\lone{\delta_1}=0$.  We start by taking a measure $\mu \in\Eone$ such
that $\int f\d\mu=\beta$ for some $\beta\not=\alpha$ which also
satisfies $\lone\mu>0$.

We now combine these two ergodic measures to find a new
(non-invariant) measure with high dimension but for which the Birkhoff
averages of $f$ tend to $\alpha$ at almost every point.  For
$\e>0$ and $N\geq 1$, we define $\Omega(\e,N)$ to be the set of $\w\in
\Sigma$ such that for all $n\geq N$ 
\begin{align}
\A_{n}f(\omega)\in B(\beta,\e), \\
\label{diameter}
\A_{n}g(\omega)\in B(\lone\mu,\e), \\ 
\label{ent} -\frac{1}{n}\log \mu([\omega_1,\ldots,\omega_{n}])\in 
    B(\hone\mu,\e). 
\end{align}
It follows from the Birkhoff Ergodic Theorem, the
Shannon--McMillan--Breiman Theorem and Egorov's Theorem that for any
fixed $\delta>0$, we can find a decreasing sequence $\e_i>0$, such that
\begin{equation}\label{muOmega}
\mu(\Omega(\e_i',i))\geq 1-\delta,
\end{equation}
where \( \lim_{i\rightarrow\infty}\e_i'=0 \)

By the uniformity of the conclusion of Lemma~\ref{conv} and since
$\var_n \A_n f$ and $\var_n \A_n g$ uniformly decrease to zero, 
we can, for $i=1,2,\dots$, choose another $\e''_i$ decreasing to zero,
so that for all $\omega\in\Sigma$ and all $n\geq i$ we have
\begin{align}
\label{varf} \var_n \A_nf(\omega)\leq\e_i'' \\
\label{varg} \var_n \A_ng(\omega)\leq\e_i'' \\
\label{diamest} \ll_n(\omega)\leq \A_ng(\omega)+\e_i''.
\end{align}

Finally, let $\e_i = \max\{ \e_i', \e_i''\}$ and let $\Omega(\e_i) =
\Omega(\e_i,i)$. Note that $\e_i$, by the construction above, is
decreasing. 
We also need another sequence of integers $\{k_i\}$ such that
$\displaystyle\lim_{i\rightarrow\infty}k_i=\infty$ but
\begin{equation}\label{dominance}
\lim_{i\rightarrow\infty} k_i\epsilon_i=0.
\end{equation}

For each $i$ we define two measures $\nu_i\in\CM(\CF^{i})$ and
$\eta_i\in\CM(\CF^{ik_i})$ where $\nu_i$ simply gives equal weight to
any cylinder containing an element of $\Omega(\epsilon_i)$ and
$\eta_i$ is simply the Dirac measure on the cylinder $[1,1,1,\dots,1]$
of length $i k_i$. For $q\in\N$ let $n_q=\sum_{i=1}^q i(1+k_i)$.
Define the probability $\eta\in \CM(\Sigma)$ to be the distribution of
a sequence of independent blocks that alternately have distribution
$\nu_i$ and $\eta_i$, for $i=1,2,3,\dots$. That is, let
\[
\eta = \otimes_{q=1}^{\infty}
\left[
  \nu_q\circ\sigma^{n_{q-1}} \otimes \eta_q\circ\sigma^{n_{q-1}+q}
\right].
\]
The measure $\eta$ is not invariant. However, the behaviour of the
Birkhoff average $A_n f(\w)$ of a continuous function $f$ for an
$\eta$-typical point $\w$ will approach the value $f(\w_1)$. This is
because the proportion of $n$'s, $1\leq n\leq N$, such that
$\sigma^n\w$ are close to $\w_1$ approaches $1$ as $N\to\infty$.
  
\begin{lemma}
  \begin{enumerate}
\item[]
  \item[1.]  For $\eta$ almost all $\omega$ we have $\displaystyle
\lim_{n\rightarrow\infty} \A_n f(\w)=\alpha$,
  \item[2.]  $\dim \eta\circ\Pi^{-1} = \frac{\hone\mu}{\lone\mu}$.
  \end{enumerate}
\end{lemma}
\begin{proof}
  Note that since $f$ is
  bounded and
  \begin{equation}\label{nqq}
  \lim_{q\rightarrow\infty}\frac{n_q-n_{q-1}}{n_q}=0
\end{equation}

  to show part 1 we can just consider the limit along the subsequence $n_q$.

  It follows from the definition of $\eta$ that for $\eta$-almost all
  $\omega$
$$\frac{\sum_{i=1}^q i(\beta-2\epsilon_i)+\sum_{i=1}^q
  k_ii(\alpha-2\epsilon_i)}{n_q}\leq\A_{n_q}f(\omega)$$
$$\leq\frac{\sum_{i=1}^qi(\beta+2\epsilon_i)
  +\sum_{i=1}^qk_ii(\alpha+2\epsilon_i)}{n_q}.$$
The first part then follows since
$\displaystyle\lim_{q\rightarrow\infty}\frac{\sum_{i=1}^qi}{n_q}=0$ and
$\displaystyle\lim_{q\rightarrow\infty}\frac{\sum_{i=i}^qk_i
  i(\alpha+\epsilon_i)}{n_q}=\alpha$.

For the second part recall that for any probability measure $\nu$ on $[0,1]$,
if for $\nu$-almost all $x$
$$\liminf_{r\rightarrow 0^+}\frac{\log{\nu(B(x,r))}}{\log r}\geq s,$$
then $\dim\nu\geq s$. (Here $B(x,r)=\{y:|y-x|<r\}$.)  We let
$\nu=\eta\circ\Pi^{-1}$ and in order to bound the ratio above we will now
consider bounds on the quantities \(-\log\eta[\omega_1,\ldots,\omega_{n}]\) and
\({n\ll_n(\omega)}\).

By using conditions \qr{ent} and \qr{dominance}, we may deduce that the
entropy of the distribution of the independent blocks
$(\w_{n_{i-1}},\w_{n_{i-1}+1},\dots,\w_{n_{i-1}+i-1})$ is at
least $\log(1-\delta) + i (\hone\mu-\e_i)$. Since we use the uniform
distribution, we may deduce that for all $\omega$
\begin{equation}
-\log\eta([\omega_1,\ldots,\omega_{n_q}])\geq
 q\,\log(1-\delta)+\sum_{i=1}^q i (\hone\mu-\e_i). \label{ent0}
\end{equation}
Note that $q\log (1-\delta)$ is of asymptotic order
$\ordo{\sum_{i=1}^q i \hone\mu}$ for large $q$.

For estimating the diameters of the cylinders, we note that
\begin{align}
\nonumber
  n_q\ll_{n_q}(\omega)&\leq \sum_{j=0}^{n_q-1} g(\sigma^j
  \omega)+n_q\epsilon_q\qquad \text{(by \eqref{diamest})}\\
\nonumber
  &\leq\sum_{i=1}^{q}\left( \sum_{j=n_{i-1}}^{n_i-1}
    g(\sigma^{j}\omega)
      +\e_ii(1+k_i)\right)\qquad \text{(since $\e_i\geq \e_q$)}\\
&=\sum_{i=1}^q\left( i\, \A_i g(\sigma^{n_{i-1}}\omega)
+(i k_i)\, \A_{ik_i} g(\sigma^{n_{i-1}+i}\w)+\e_ii(1+k_i)\right).
\label{finnn}
\end{align}
Since the $[\sigma^{n_{i-1}}\omega]_i$ contains an $\w$ from
  $\Omega(\e_i)$ and since $\sigma^{n_{i-1}+i}\w \in [\w_0]_{ik_i}$
  we obtain, by \eqref{diameter} and \eqref{varg}, that \qr{finnn} is
  less than
\begin{align}
\nonumber
  &\sum_{i=1}^q
  \left(i(\lone\mu+\epsilon_i + \var_i g)+ik_i(A_{ik_i}g(\omega_0)+
    \var_{ik_i} g)+\e_ii(1+k_i)\right) \\
  \leq&\sum_{i=1}^q i(\lone\mu+ 3\e_i + 2\e_i k_i). 
\label{diam0}
\end{align}
Recall that by condition (\ref{dominance}) $\lim_{i\rightarrow\infty}
\epsilon_i k_i=0$.

We now fix $q>0$ and choose $r$ such that
$$\exp\left(-\sum_{i=1}^q i(\lone\mu+ 3\e_i + 2\e_i k_i)\right)\geq r
>\exp\left(-\sum_{i=1}^{q+1} i(\lone\mu+ 3\e_i + 2\e_i k_i)\right).$$
Consider a ball $B(x,r)$ where $\Pi\omega=x$. It follows from the 
choice of $r$ above that $B(x,r)$ can intersect at most $3$ cylinders 
of length $n_q$ which carry positive $\eta$-measure. Thus, again using the 
definition of $r$ together with (\ref{ent0}), we obtain
$$\frac{\log\nu(B(x,r))}{\log r}\geq\frac{\log 3+q\,\log(1-\delta)
+\sum_{i=1}^q i (\hone\mu-\e_i)}{\sum_{i=1}^{q} i(\lone\mu+ 3\e_i 
+ 2\e_i k_i)+(q+1)(\lone\mu+ 3\e_{q+1} + 2\e_{q+1} k_{q+1})}.$$
Using condition (\ref{dominance}) we can observe that this is of the form
$$\frac{-\ordo{q^2}+\sum_{i=1}^q i (\hone\mu-\ordo{1})}{\ordo{q^2}+
\sum_{i=1}^{q} i(\lone\mu+ \ordo{1})}.$$
 Since $r\to 0$ as $q\rightarrow\infty$, this means for $\nu$ 
almost all $x$
$$\lim_{r\to 0+}\frac{\log\nu(B(x,r))}{\log r}\geq
\frac{\hone\mu}{\lone\mu},$$
which completes the proof.
\end{proof}
\noindent {\bf Acknowledgment.} The second author was supported by the 
EPSRC. The third author was supported by a STINT Fellowship to visit Amherst 
College during the Fall Semester of 2007.



\begin{thebibliography}{999}
\bibitem{BI} L. Barreira and G. Iommi, {\em Parabolic horseshoes}, preprint.
\bibitem{BS} L.\ Barreira and B.\ Saussol, {\em Variational principles
    and mixed multifractal spectra}, Trans.\ Amer.\ Math.\ Soc.\
  {\bf 353} (2001), no.\ 10, 3919--3944.
\bibitem{B} W. Byrne,  Multifractal Analysis of Parabolic Rational
  Maps, Phd Thesis, The University of North Texas.
\bibitem{GR} K.\ Gelfert and M.\ Rams, {\em Multifractal analysis of
    Lyapunov exponents of parabolic iterated function systems},
  preprint at http://www.impan.gov.pl/$\%$7Erams/mfanuh20.ps.gz
\bibitem{HMU}
P. Hanus, D. Mauldin and M. Urbanski,
{\em Thermodynamic formalism and multifractal analysis of conformal infinite iterated function systems.}, 
Acta Math. Hungar. \textbf{96} (2002), no. 1-2, 27--98.


\bibitem{HR} F.\ Hofbauer and P.\ Raith, {\em The Hausdorff
    dimension of an ergodic invariant measure for a piecewise
    monotonic map of the interval}, Canad.\ Math.\ Bull.\ {\bf 35}
  (1992), no.\ 1, 84--98.
\bibitem{Katok} A. Katok, {\em Lyapunov exponents, entropy and
    periodic orbits for diffeomorphisms}, Inst. Hautes Études
  Sci. Publ. Math.\ \textbf{51} (1980), 137--173.


\bibitem{KS}
 M. Kesseb\"{o}hmer, B. Stratmann. {\em A multifractal formalism for 
growth rates and applications to geometrically finite Kleinian groups}, 
Ergodic Theory and Dynam.\ Systems {\bf 24} (2004), no. 1, 141--170. 
\bibitem{N}K. Nakaishi, Multifractal formalism for some parabolic
  maps, Ergodic theory and dynamical systems, {\bf 20} (2000),
  843-857.
\bibitem{O} L.\ Olsen, {\em Multifractal analysis of divergence
    points of deformed measure theoretical Birkhoff averages}, J.\
  Math.\ Pures Appl.\ {\bf 82} (2003), 1591--1649.
\bibitem{P} Y.\ Pesin, Dimension Theory in Dynamical Systems,
  Contemporary Views and Applications, Chicago Lectures in
  Mathematics. University of Chicago Press, Chicago 1997.
\bibitem{PW} Y.\ Pesin and H.\ Weiss, {\em The multifractal analysis
    of Birkhoff averages and large deviations}, Global analysis of
  dynamical systems, 419--431, Inst.\ Phys., Bristol, 2001.
\bibitem{VerbitskyTakens}
F.\ Takens and E.\ Verbitskiy, {\em On the variational principle for the 
topological entropy of certain non-compact sets},
Ergodic Theory and Dynam.\ Systems {\bf 23} (2003), no.\ 1, 317--348.
\bibitem{U}M. Urba\'nski, Parabolic Cantor sets. Fund.\ Math.\ {\bf 151}
(1996), no.\ 3, 241--277.
\bibitem{W} P.\ Walters, An Introduction to Ergodic Theory,
  Springer, 1982.
 
\end{thebibliography}
\end{document}